\documentclass{article}
\usepackage{hyperref}
\usepackage{amsmath}
\usepackage{amsthm}
\usepackage{amssymb}
\usepackage{stmaryrd}  % for ovee

\theoremstyle{definition}
\newtheorem{Def}{Definition}
\newtheorem{definition}[Def]{Definition}
\newtheorem{theorem}[Def]{Theorem}
\newtheorem{proposition}[Def]{Proposition}
\newtheorem{lemma}[Def]{Lemma}

\newtheorem{corollary}[Def]{Corollary}
\newtheorem{example}[Def]{Example}

\usepackage{tikz-cd}
\newcommand{\sem}[1]{\ensuremath{\llbracket #1 \rrbracket}}

\newcommand{\G}{G}
\newcommand{\M}{M}
\newcommand{\g}{g}
\newcommand{\m}{m}
\newcommand{\K}{\ctx{K}}
\newcommand{\ctx}[1]{\mathbb{#1}}
\newcommand{\I}{\models}  % Incidence relation of formal context 

\newcommand{\pset}{\mathbb{P}}

\newcommand{\Sx}[1]{\underline{#1}}

\newcommand{\clat}{\mathbb{B}}
\newcommand{\op}{\ensuremath{\mathrm{\rm op}}}

\newcommand{\TV}{\mathsf{2}} 
\newcommand{\fcotimes}{\otimes}
\newcommand{\fco}{\fcotimes}

\newcommand{\xtimes}{\underline{\otimes}}

\newcommand{\myemph}[1]{\emph{#1}}
\newcommand{\cl}[1]{\overline{#1}}
\newcommand{\cls}[1]{\cl{#1}}
\usepackage{stackengine} 
\newcommand{\twedge}{\owedge}

\newcommand{\tvee}{\ovee}

\newcommand{\cat}[1]{\ensuremath{\mathbf{#1}}}
\newcommand{\catC}{\cat{C}}
\newcommand{\Cxt}{\cat{Cxt}}
\newcommand{\Ctx}{\Cxt}

\newcommand{\Rel}{\cat{Rel}}
\newcommand{\InfLat}{\cat{InfLat}}
\newcommand{\SupLat}{\cat{SupLat}}
\newcommand{\CompLat}{\cat{CLat}}

\newcommand{\id}[1]{\ensuremath{\mathrm{id}_{#1}}}
\newcommand{\Bonds}{\cat{Bonds}}
\newcommand{\ChuCors}{\cat{ChuCors}}

\newcommand{\Ictx}{\ctx{I}}

\newcommand{\Icxt}{\Ictx}

\newcommand{\hilbH}{\mathcal{H}} %Hilbert space
\newcommand{\ve}{\varepsilon}

\newcommand{\lotimes}{\boxtimes}

\newcommand{\ua}{\bullet}

\newcommand{\discard}[1]{\ensuremath{\tinygroundflipnew_{#1}}}

\usepackage{tikz,xypic}
\usetikzlibrary{decorations.pathreplacing,decorations.markings,arrows.meta,backgrounds}
\pgfdeclarelayer{edgelayer}
\pgfdeclarelayer{nodelayer}
\pgfsetlayers{background,edgelayer,nodelayer,main}
\tikzstyle{whitedot}=[circle, draw=black, fill=white, inner sep=.4ex]
\tikzstyle{none}=[inner sep=0mm]

\usepackage{tikz,xypic}

\usetikzlibrary{decorations.pathreplacing,decorations.markings,arrows.meta,backgrounds}
\pgfdeclarelayer{edgelayer}
\pgfdeclarelayer{nodelayer}
\pgfsetlayers{background,edgelayer,nodelayer,main}

\tikzstyle{cdot}=[circle, draw=black, fill=black!25, inner sep=.4ex] %Chris dot style 
\tikzstyle{bigdot}=[dot, inner sep=0pt]
\tikzstyle{whitedot}=[circle, draw=black, fill=white, inner sep=.4ex]
\tikzstyle{greydot}=[circle, draw=black, fill=black!25, inner sep=.4ex] %Added by Sean
\tikzstyle{blackdot}=[circle, draw=black, fill=black, inner sep=.4ex]
\tikzset{arrow/.style={decoration={
    markings,
    mark=at position #1 with \arrow{>[length=2pt, width=3pt]}},
    postaction=decorate},
    reverse arrow/.style={decoration={
    markings,
    mark=at position #1 with {{\arrow{<[length=2pt, width=3pt]}}}},
    postaction=decorate}
}

\newenvironment{pic}[1][] {\begin{aligned}\begin{tikzpicture}[scale=2.0, font=\tiny,#1]}{\end{tikzpicture}\end{aligned}} %Hard coded picture 

\newif\ifvflip\pgfkeys{/tikz/vflip/.is if=vflip}
\newif\ifhflip\pgfkeys{/tikz/hflip/.is if=hflip}
\newif\ifhvflip\pgfkeys{/tikz/hvflip/.is if=hvflip}

\newenvironment{picc}[1][]
{\begin{aligned}\begin{tikzpicture}[font=\tiny,#1]}
{\end{tikzpicture}\end{aligned}}

\newlength\minimummorphismwidth
\newlength\stateheight
\newlength\minimumstatewidth
\newlength\connectheight
\tikzset{colour/.initial=white}

\tikzstyle{pure}=[line width=.7pt]

\makeatletter

\pgfdeclareshape{groundd}
{
    \savedanchor\centerpoint
    {
        \pgf@x=0pt
        \pgf@y=0pt
    }
    \anchor{center}{\centerpoint}
    \anchorborder{\centerpoint}

    \anchor{north}
    {
        \pgf@x=0pt
        \pgf@y=0.16\stateheight
    }
    \anchor{south}
    {
        \pgf@x=0pt
        \pgf@y=0pt
    }
    \saveddimen\overallwidth
    {
        \pgfkeysgetvalue{/pgf/minimum width}{\minwidth}
        \pgf@x=\minimumstatewidth
        \ifdim\pgf@x<\minwidth
            \pgf@x=\minwidth
        \fi
    }
    \backgroundpath
    {
        \begin{pgfonlayer}{main} %Sean: this used to be 'foreground', swtich back if trouble
        \pgfsetstrokecolor{black}
        \pgfsetlinewidth{1.25pt}
        \ifhflip
            \pgftransformyscale{-1}
        \fi
        \pgftransformscale{0.5}
        \pgfpathmoveto{\pgfpoint{-0.5*\overallwidth}{0pt}}
        \pgfpathlineto{\pgfpoint{0.5*\overallwidth}{0pt}}
        \pgfpathmoveto{\pgfpoint{-0.33*\overallwidth}{0.33*\stateheight}}
        \pgfpathlineto{\pgfpoint{0.33*\overallwidth}{0.33*\stateheight}}
        \pgfpathmoveto{\pgfpoint{-0.16*\overallwidth}{0.66*\stateheight}}
        \pgfpathlineto{\pgfpoint{0.16*\overallwidth}{0.66*\stateheight}}
        \pgfpathmoveto{\pgfpoint{-0.02*\overallwidth}{\stateheight}}
        \pgfpathlineto{\pgfpoint{0.02*\overallwidth}{\stateheight}}
        \pgfusepath{stroke}
        \end{pgfonlayer}
    }
}

\usepackage{tikz,xypic}
\usetikzlibrary{decorations.pathreplacing,decorations.markings,arrows.meta,backgrounds,shapes}
\usetikzlibrary{circuits.ee.IEC}
\pgfdeclarelayer{edgelayer}
\pgfdeclarelayer{nodelayer}
\pgfsetlayers{background,edgelayer,nodelayer,main}
\tikzstyle{none}=[inner sep=0mm]
\tikzstyle{every loop}=[]
\tikzstyle{mark coordinate}=[inner sep=0pt,outer sep=0pt,minimum size=3pt,fill=black,circle]

\tikzset{arrow/.style={decoration={
    markings,
    mark=at position #1 with \arrow{>[length=2pt, width=3pt]}},
    postaction=decorate},
    reverse arrow/.style={decoration={
    markings,
    mark=at position #1 with {{\arrow{<[length=2pt, width=3pt]}}}},
    postaction=decorate}
}

\tikzstyle{upground}=[circuit ee IEC,thick,ground,rotate=90,scale=1.5]
\tikzstyle{upgroundwhite}=[circuit ee IEC,thick,ground,rotate=90,scale=1.5, fill=white]
\tikzstyle{downground}=[circuit ee IEC,thick,ground,rotate=-90,scale=1.5]
\tikzstyle{downgroundnorm}=[circuit ee IEC,thick,ground,rotate=-90,scale=1.5, fill=white]
\newcommand{\mapminh}{5mm} % Aleks and Bob have 6mm
\newcommand{\stateminh}{5mm}
\newcommand{\maplw}{0.7pt} % Alesk and Bob previously just left as default
\newcommand{\stateshift}{-0.2pt}%{1pt}
\newcommand{\effectshift}{-0.2pt}%{1pt}

\tikzstyle{box}=[map]
\tikzstyle{medium box}=[medium map]
\tikzstyle{dot}=[inner sep=0mm,minimum width=2mm,minimum height=2mm,draw,shape=circle]  
\tikzstyle{black dot}=[dot,fill=black]
\tikzstyle{white dot}=[dot,fill=white,,text depth=-0.2mm]
\tikzstyle{grey dot}=[dot,fill=black!25] %Added by Sean

\tikzstyle{corner1}=[box,fill=white, font=\footnotesize] %
\tikzstyle{corner2}=[dot,fill=white, font=\footnotesize] %
\tikzstyle{corner3}=[dot,fill=black!25, font=\footnotesize] %
\tikzstyle{corner4}=[dot,fill=black, font=\footnotesize] %

\tikzstyle{scalar}=[circle,draw,inner sep=2pt, line width=\maplw] %,font=\small] %Consider changing to a circle!

\usetikzlibrary{shapes.misc, positioning}

\tikzset{stateshape/.style={append after command={
   \pgfextra
        \draw[sharp corners, fill=white, line width = \maplw]% 
    (\tikzlastnode.west)% 
    [rounded corners=0pt] |- (\tikzlastnode.north)% 
    [rounded corners=0pt] -| (\tikzlastnode.east)% 
    [rounded corners=5pt] |- (\tikzlastnode.south)% 
    [rounded corners=5pt] -| (\tikzlastnode.west);
   \endpgfextra}}}

\tikzset{effectshape/.style={append after command={
   \pgfextra
        \draw[sharp corners, fill=white, line width = \maplw]% 
    (\tikzlastnode.west)% 
    [rounded corners=0pt] |- (\tikzlastnode.south)% 
    [rounded corners=0pt] -| (\tikzlastnode.east)% 
    [rounded corners=5pt] |- (\tikzlastnode.north)% 
    [rounded corners=5pt] -| (\tikzlastnode.west);
   \endpgfextra}}}

 \tikzstyle{map}=[draw,shape=rectangle, inner sep=2pt,minimum height=\mapminh, minimum width=7mm,fill=white]

\tikzstyle{point}=[stateshape,inner sep=2pt, minimum width=6mm, minimum height=\stateminh, yshift=\stateshift]
\tikzstyle{copoint}=[effectshape,inner sep=.2pt, minimum width=6mm, minimum height=\stateminh, yshift=-\effectshift]

\tikzstyle{wide point}=[point, minimum width=12mm]
\tikzstyle{wide copoint}=[copoint, minimum width=12mm]

\tikzstyle{decomp}=[fill=white,draw,shape=isosceles triangle,shape border rotate=-90,isosceles triangle stretches=true,inner sep=0pt,minimum width=0.75cm,minimum height=4mm,yshift=-0.0mm]

\tikzstyle{decompwide}=[fill=white,draw,shape=isosceles triangle,shape border rotate=-90,isosceles triangle stretches=true,inner sep=0pt,minimum width=1.5cm,minimum height=4mm,yshift=-0.0mm]

\tikzstyle{decompflip}=[fill=white,draw,shape=isosceles triangle,shape border rotate=90,isosceles triangle stretches=true,inner sep=0pt,minimum width=0.75cm,minimum height=4mm,yshift=-0.0mm]

\tikzstyle{decompwideflip}=[fill=white,draw,shape=isosceles triangle,shape border rotate=90,isosceles triangle stretches=true,inner sep=0pt,minimum width=1.5cm,minimum height=4mm,yshift=-0.0mm]

 \tikzstyle{map}=[draw,shape=rectangle, inner sep=2pt,minimum height=\mapminh, minimum width=7mm,fill=white, line width = \maplw]

\tikzstyle{medium map} = [map, minimum width = 12mm] 
\tikzstyle{semilarge map} = [map, minimum width = 15mm] 
\tikzstyle{large map} = [map, minimum width = 18mm] 

\tikzstyle{kpoint} =[point]
\tikzstyle{kpointadj} =[copoint]
\tikzstyle{kpointconj}=[dagpointconj] %Need daggers for conjugates
\makeatletter
\newcommand{\boxshape}[3]{%
\pgfdeclareshape{#1}{
\inheritsavedanchors[from=rectangle] % this is nearly a rectangle
\inheritanchorborder[from=rectangle]
\inheritanchor[from=rectangle]{center}
\inheritanchor[from=rectangle]{north}
\inheritanchor[from=rectangle]{south}
\inheritanchor[from=rectangle]{west}
\inheritanchor[from=rectangle]{east}
\backgroundpath{% this is new
\southwest \pgf@xa=\pgf@x \pgf@ya=\pgf@y
\northeast \pgf@xb=\pgf@x \pgf@yb=\pgf@y

\@tempdima=#2
\@tempdimb=#3

\pgfpathmoveto{\pgfpoint{\pgf@xa - 5pt + \@tempdima}{\pgf@ya}}
\pgfpathlineto{\pgfpoint{\pgf@xa - 5pt - \@tempdima}{\pgf@yb}}
\pgfpathlineto{\pgfpoint{\pgf@xb + 5pt + \@tempdimb}{\pgf@yb}}
\pgfpathlineto{\pgfpoint{\pgf@xb + 5pt - \@tempdimb}{\pgf@ya}}
\pgfpathlineto{\pgfpoint{\pgf@xa - 5pt + \@tempdima}{\pgf@ya}}
\pgfpathclose
}
}}

\boxshape{NEbox}{0pt}{3pt} %My setting of 3pt
\boxshape{SEbox}{0pt}{-3pt}
\boxshape{NWbox}{3pt}{0pt}
\boxshape{SWbox}{-3pt}{0pt}
\boxshape{rec-box}{0pt}{0pt}
\makeatother

\tikzstyle{cloud}=[shape=cloud,draw,minimum width=1.5cm,minimum height=1.5cm]

\tikzstyle{dagmap}=[draw,shape=NEbox,inner sep=2pt,minimum height=\mapminh,fill=white, line width = \maplw] %
\tikzstyle{dashedmap}=[draw,dashed,shape=NEbox,inner sep=2pt,minimum height=\mapminh,fill=white, line width = \maplw]
\tikzstyle{mapdag}=[draw,shape=SEbox,inner sep=2pt,minimum height=\mapminh,fill=white, line width = \maplw]
\tikzstyle{mapadj}=[draw,shape=SEbox,inner sep=2pt,minimum height=\mapminh,fill=white, line width = \maplw]
\tikzstyle{maptrans}=[draw,shape=SWbox,inner sep=2pt,minimum height=\mapminh,fill=white, line width = \maplw]
\tikzstyle{mapconj}=[draw,shape=NWbox,inner sep=2pt,minimum height=\mapminh,fill=white, line width = \maplw]

\tikzstyle{medium dagmap}=[draw,shape=NEbox,inner sep=2pt,minimum height=\mapminh,fill=white,minimum width=7mm, line width = \maplw]
\tikzstyle{semilarge dagmap}=[draw,shape=NEbox,inner sep=2pt,minimum height=\mapminh,fill=white,minimum width=9.5mm, line width = \maplw]
\tikzstyle{large dagmap}=[draw,shape=NEbox,inner sep=2pt,minimum height=\mapminh,fill=white,minimum width=12mm, line width = \maplw]
\makeatletter

\pgfdeclareshape{cornerpoint}{
\inheritsavedanchors[from=rectangle] % this is nearly a rectangle
\inheritanchorborder[from=rectangle]
\inheritanchor[from=rectangle]{center}
\inheritanchor[from=rectangle]{north}
\inheritanchor[from=rectangle]{south}
\inheritanchor[from=rectangle]{west}
\inheritanchor[from=rectangle]{east}
\backgroundpath{% this is new
\southwest \pgf@xa=\pgf@x \pgf@ya=\pgf@y
\northeast \pgf@xb=\pgf@x \pgf@yb=\pgf@y

\pgfmathsetmacro{\pgf@shorten@left}{\pgfkeysvalueof{/tikz/shorten left}}
\pgfmathsetmacro{\pgf@shorten@right}{\pgfkeysvalueof{/tikz/shorten right}}

\pgfpathmoveto{\pgfpoint{0.5 * (\pgf@xa + \pgf@xb)}{\pgf@ya - 5pt}}
\pgfpathlineto{\pgfpoint{\pgf@xa - 8pt + \pgf@shorten@left}{\pgf@yb - 1.5 * \pgf@shorten@left}}
\pgfpathlineto{\pgfpoint{\pgf@xa - 8pt + \pgf@shorten@left}{\pgf@yb}}
\pgfpathlineto{\pgfpoint{\pgf@xb + 8pt - \pgf@shorten@right}{\pgf@yb}}
\pgfpathlineto{\pgfpoint{\pgf@xb + 8pt - \pgf@shorten@right}{\pgf@yb - 1.5 * \pgf@shorten@right}}
\pgfpathclose
}
}

\pgfdeclareshape{cornercopoint}{
\inheritsavedanchors[from=rectangle] % this is nearly a rectangle
\inheritanchorborder[from=rectangle]
\inheritanchor[from=rectangle]{center}
\inheritanchor[from=rectangle]{north}
\inheritanchor[from=rectangle]{south}
\inheritanchor[from=rectangle]{west}
\inheritanchor[from=rectangle]{east}
\backgroundpath{% this is new
\southwest \pgf@xa=\pgf@x \pgf@ya=\pgf@y
\northeast \pgf@xb=\pgf@x \pgf@yb=\pgf@y

\pgfmathsetmacro{\pgf@shorten@left}{\pgfkeysvalueof{/tikz/shorten left}}
\pgfmathsetmacro{\pgf@shorten@right}{\pgfkeysvalueof{/tikz/shorten right}}

\pgfpathmoveto{\pgfpoint{0.5 * (\pgf@xa + \pgf@xb)}{\pgf@yb + 5pt}}
\pgfpathlineto{\pgfpoint{\pgf@xa - 8pt + \pgf@shorten@left}{\pgf@ya + 1.5 * \pgf@shorten@left}}
\pgfpathlineto{\pgfpoint{\pgf@xa - 8pt + \pgf@shorten@left}{\pgf@ya}}
\pgfpathlineto{\pgfpoint{\pgf@xb + 8pt - \pgf@shorten@right}{\pgf@ya}}
\pgfpathlineto{\pgfpoint{\pgf@xb + 8pt - \pgf@shorten@right}{\pgf@ya + 1.5 * \pgf@shorten@right}}
\pgfpathclose
}
}

\makeatother

\pgfkeyssetvalue{/tikz/shorten left}{0pt}
\pgfkeyssetvalue{/tikz/shorten right}{0pt}

\tikzstyle{dagpoint common}=[draw,fill=white,inner sep=1pt, line width = \maplw, minimum height = 4mm, yshift=1.2pt] %SEAN: I ADDED THIS yshift to make it line up with maps. 1pt seems to work.
\tikzstyle{dagpoint sc}=[shape=cornerpoint,dagpoint common]
\tikzstyle{dagpoint adjoint sc}=[shape=cornercopoint,dagpoint common]
\tikzstyle{dagpoint}=[shape=cornerpoint,shorten left=4pt,dagpoint common]
\tikzstyle{dagpointadj}=[shape=cornercopoint,shorten left=5pt,dagpoint common]
\tikzstyle{dagpointconj}=[shape=cornerpoint,shorten right=5pt,dagpoint common]
\tikzstyle{dagpointtrans}=[shape=cornercopoint,shorten right=5pt,dagpoint common]
\tikzstyle{dagpointsymm}=[shape=cornerpoint,shorten left=5pt,shorten right=5pt,dagpoint common]

\tikzstyle{widedagpoint}=[dagpoint, minimum width=1 cm, inner sep=2pt]%, text depth=-0.7 mm]
\tikzstyle{widedagpointadj}=[dagpointadj, minimum width=1 cm, inner sep=2pt]%, text depth=0.7 mm]

\tikzstyle{every picture}=[baseline=-0.25em,scale=0.5]
\tikzstyle{label}=[font=\footnotesize,text height=1ex, text depth=0.15ex]

\usetikzlibrary[shapes]

\tikzset{
sidetriangle/.style = {regular polygon, regular polygon sides = 3, aspect = 1, shape border rotate = 90, draw, inner sep = 0, minimum width = 1.2cm}
}

\tikzset{
isoc/.style = {shape=isosceles triangle, shape border rotate = 180, isosceles triangle stretches = true, minimum width = 1.2cm, minimum height= 1.5cm, inner sep = 0.3}}

\tikzset{
coarse/.style = {shape = circle, fill = white, draw, inner sep = 0, minimum width =0.125cm}
}
\tikzset{
coarsesymbol/.style = {shape = circle, fill = white, inner sep = -0.7, minimum width = 0.125cm}
}
\tikzstyle{sidetriangle2}=[sidetriangle, minimum width = 2cm, fill=white]
\tikzstyle{sideisocsmall}]=[style=isoc, minimum width = 1cm, minimum height = 0.8cm, draw, fill=white, font=\Large]
\tikzstyle{sideisoc}]=[style=isoc, minimum width = 2cm, draw, fill=white, font=\Large]
\tikzstyle{sideisocmid}]=[style=isoc, minimum width = 2.5cm, draw, fill=white, font=\Large]
\tikzstyle{sideisocmedium}]=[style=isoc, minimum width = 3cm, draw, fill=white, font=\Large]

\newcommand{\tinygroundflipnew}{
\smash{
{\hspace{-3pt}
\ensuremath{
\begin{picc}[yscale=-1.0] 
    \node[upground, xscale=0.8, yscale=-0.7] (1) at (0,0.10) {};
    \draw (0,-0.03) to (0,-0.31);
\end{picc}
}\hspace{-1pt}}}}

\tikzstyle{label}=[font=\footnotesize,text height=1ex, text depth=0.15ex]

\tikzstyle{box}=[map]
\tikzstyle{medium box}=[medium map]
\tikzstyle{dot}=[inner sep=0mm,minimum width=2mm,minimum height=2mm,draw,shape=circle]  
\tikzstyle{black dot}=[dot,fill=black]
\tikzstyle{white dot}=[dot,fill=white,,text depth=-0.2mm]
\tikzstyle{grey dot}=[dot,fill=black!25] %Added by Sean

\tikzstyle{corner1}=[box,fill=white, font=\footnotesize] %
\tikzstyle{corner2}=[dot,fill=white, font=\footnotesize] %
\tikzstyle{corner3}=[dot,fill=black!25, font=\footnotesize] %
\tikzstyle{corner4}=[dot,fill=black, font=\footnotesize] %

\tikzstyle{scalar}=[circle,draw,inner sep=2pt, line width=\maplw] %,font=\small] %Consider changing to a circle!

\usetikzlibrary{shapes.misc, positioning}

\tikzset{stateshape/.style={append after command={
   \pgfextra
        \draw[sharp corners, fill=white, line width = \maplw]% 
    (\tikzlastnode.west)% 
    [rounded corners=0pt] |- (\tikzlastnode.north)% 
    [rounded corners=0pt] -| (\tikzlastnode.east)% 
    [rounded corners=5pt] |- (\tikzlastnode.south)% 
    [rounded corners=5pt] -| (\tikzlastnode.west);
   \endpgfextra}}}

\tikzset{effectshape/.style={append after command={
   \pgfextra
        \draw[sharp corners, fill=white, line width = \maplw]% 
    (\tikzlastnode.west)% 
    [rounded corners=0pt] |- (\tikzlastnode.south)% 
    [rounded corners=0pt] -| (\tikzlastnode.east)% 
    [rounded corners=5pt] |- (\tikzlastnode.north)% 
    [rounded corners=5pt] -| (\tikzlastnode.west);
   \endpgfextra}}}

 \tikzstyle{map}=[draw,shape=rectangle, inner sep=2pt,minimum height=\mapminh, minimum width=7mm,fill=white]

\tikzstyle{point}=[stateshape,inner sep=2pt, minimum width=6mm, minimum height=\stateminh, yshift=\stateshift]
\tikzstyle{copoint}=[effectshape,inner sep=.2pt, minimum width=6mm, minimum height=\stateminh, yshift=-\effectshift]

\tikzstyle{wide point}=[point, minimum width=12mm]
\tikzstyle{wide copoint}=[copoint, minimum width=12mm]

\title{\textbf{Monoidal Categories for Formal Concept Analysis}}

\title{Monoidal Categories for Formal Concept Analysis}
\author{Sean Tull}
\date{\vspace{-8pt}\small{Cambridge Quantum Computing \\ \url{sean.tull@cambridgequantum.com}}}
\begin{document}
\maketitle

\begin{abstract}
We investigate monoidal categories of formal contexts in which states correspond to formal concepts. In particular we examine the category of bonds or Chu correspondences between contexts, which is known to be equivalent to the *-autonomous category of complete sup-lattices. We show that a second monoidal structure exists on both categories, corresponding to the direct product of formal contexts defined by Ganter and Wille, and discuss the use of these categories as compositional models of meaning.
\end{abstract}

\section*{Introduction}

\emph{Formal concept analysis (FCA)} is a highly successful framework for reasoning about collections of objects and their properties, initiated by Wille \cite{wille1992concept}. 
Starting from a system described by a \emph{formal context} of \emph{objects} and the \emph{attributes} these attain, the central feature of FCA is the construction of its hierarchy of \emph{formal concepts}, which form a complete lattice known as the \emph{concept lattice}. FCA has found many successful applications in semantics, including data mining, machine learning, the semantic web, and linguistics \cite{ganter1999formal,ganter2005formal}.

A more recently developed framework is that of \emph{Categorical Distributional Compositional Models of Meaning (DisCo)}, initiated by Coecke, Clark and Sadrzadeh \cite{coecke2010mathematical}. Typically drawing on Lambek's theory of \emph{pregroup grammars} \cite{lambek2008word}, this provides a structured recipe for deriving the meaning of a sentence in terms of the meaning of its individual words when these exist in an autonomous category. More generally, Delpeuch has extended the framework to any monoidal category \cite{delpeuch2014autonomization}. 

Though vector spaces are most commonly used, more semantic categories have recently been explored in the DisCo framework, including the use of density matrices for word meanings \cite{balkir2015distributional}, and convex relational spaces modelling Gardenfors' framework of \emph{conceptual spaces} \cite{gardenfors2004conceptual,bolt2019interacting}.

In this work, we investigate monoidal categories of formal contexts, to serve as new models of meaning in frameworks for compositional semantics such as DisCo.  Conversely, one may hope that category theory may provide new tools for FCA, as argued by Mori \cite{mori2008chu} and Pavlovic \cite{pavlovic2012quantitative,pavlovic2020nucleus}.

Since word meanings in the DisCo formalism are represented by states, we wish to consider categories whose objects are formal contexts and states correspond to their formal concepts. However, beyond this there is freedom in both our choice of morphism and tensor product of formal contexts, and several have been proposed for each \cite{wille1985tensorial,ganter1999formal,krotzsch2005morphisms,mori2008chu,erne2014categories}. 

The morphisms we consider are equivalent to the notion of \emph{bond} between formal contexts introduced by Ganter and Wille \cite{ganter1999formal}, which Mori has studied in detail via the equivalent notion of \emph{Chu correspondence} \cite{mori2008chu}, and which we show also coincide with the morphisms of contexts studied by Moshier \cite{moshier2016relational}. These form our category of interest $\Cxt$. Taking the concept lattice is known to provide an equivalence of categories between $\Cxt$ and the category $\SupLat$ of complete sup-lattices. Since the latter is known to have a {*-autonomous} structure given by the tensor product $\lotimes$ of sup-lattices, this yields the {*-autonomous} monoidal structure $(\Cxt, \lotimes)$ described by Mori in \cite{mori2008chu}. 

 However, one may prefer a tensor structure motivated by formal concepts themselves, rather than lattices. Wille has in fact introduced a notion of \emph{direct product} of formal contexts \cite{wille1985tensorial}. Here we show these provide an alternative symmetric monoidal structure $(\Cxt, \fcotimes)$. Wille has also shown the direct product to correspond to an alternative tensor product $\fcotimes$ of complete lattices. We extend Wille's results to merely sup-complete homomorphisms, to show that this tensor in fact provides a second symmetric monoidal structure $(\SupLat, \fcotimes)$ which makes the concept lattice a monoidal equivalence. 
In summary then, for each of the corresponding tensors $\xtimes \in \{\fcotimes, \lotimes \}$ on $\Cxt$ and $\SupLat$, taking concept lattices provides an equivalence of symmetric monoidal categories $(\Cxt,\xtimes) \simeq (\SupLat, \xtimes)$. Here we briefly discuss the potential use of each monoidal structure in the DisCo framework, which would be desirable to explore in future work.

\paragraph{Outline}
In Section \ref{sec:FCA} we introduce the basics of formal concept analysis. In Section \ref{sec:Cxt} we describe the category $\Cxt$ of formal contexts, giving several equivalent definitions of its morphisms. Section \ref{sec:Cxt-Mon} introduces two symmetric monoidal structures on $\Cxt$. In Section \ref{sec:lattices} we describe the monoidal equivalences $\Cxt \simeq \SupLat$ for two corresponding tensors of sup-lattices. Finally in Section \ref{sec:applications} we describe applications to the DisCo framework. 

\paragraph{Related work}
Our definition of $\Cxt$ essentially comes from the `continuous extent correspondences' of the article \cite{mori2008chu} where the equivalent category of Chu correspondences and its relation with bonds and sup-lattices, and the  *-autonomous structure $\lotimes$,  are studied. Section \ref{sec:Cxt} provides an alternative presentation of this category, and a new equivalence with the category of \cite{moshier2016relational}. Our main new results are the definition of the concept tensors $\fcotimes$ on $\Cxt$ and $\SupLat$ in Sections \ref{sec:Cxt-Mon} and \ref{subsec:concept-suplat}.

% \paragraph{Acknowledgments} We thank Bob Coecke for suggesting the exploration of FCA models for DisCo. 

\section{Formal Concept Analysis} \label{sec:FCA}

Let us now introduce the basic ingredients of Formal Concept Analysis (FCA). Throughout we follow the presentation of \cite{ganter1999formal}.

\begin{definition}
A \emph{formal context} is a tuple 
\[
\K = (\G, \M, \I)
\]
consisting of a set $\G$ of \emph{objects}, a set $\M$ of \emph{attributes}, and a relation $\I \subseteq \G \times \M$. For each $\g \in \G$ and $\m \in \M$, whenever $\I(\g,\m)$ we instead write $\g \I \m$ and say that \emph{the object $\g$ has the attribute $\m$}. 
\end{definition}

More generally, for any such context $\K$, for any subsets $A \subseteq \G$ and $B \subseteq \M$ we write $A \I B$ whenever $a \I b$ for all $a \in A$ and $b \in B$. We define 
\begin{align*}
A' &:= \{\m \in \M \mid a \I \m \   \forall a \in A \} \subseteq \M \\ 
B' &:= \{ \ \g \in \G   \ \mid g \I b \ \  \  \forall b \in B \} \subseteq \G
\end{align*}
We then have $A \subseteq B' \iff B \subseteq A'$. This means that the mappings $A \mapsto A'$ and $B \mapsto B'$ form a \emph{Galois connection} between the partially ordered sets $\pset(\G)$ and $\pset(\M)$, or in other words an adjunction
\[
\begin{tikzcd}
\pset(\G)^\op \arrow[rr,"(-)'", bend left=15] & \bot & \pset(\M)
\arrow[ll,"(-)'", bend left=15]
\end{tikzcd}
\]
As a result we obtain (idempotent and order-preserving) \emph{closure operators} on $\pset(\G)$ and $\pset(\M)$ given by 
$A \mapsto \cls{A} := A''$ and $B \mapsto \cls{B} := B''$. For each subset $A$ of $\G$ we call $\cls{A}$ the \emph{closure} of $A$, and say $A$ is \emph{closed} when $A = \cls{A}$, and similarly for $B \subseteq \M$. For any $g \in \G$ we define $g' = \{g\}'$ and $\cl{g}=\cl{\{g\}}$, and similarly for $m \in M$. We may now define concepts themselves.

\begin{definition}
A \emph{(formal) concept} of a context $\K$ is a pair $(A, B)$ where $A \subseteq \G$ and $B \subseteq \M$, satisfying $A = B'$ and $B = A'$. We call $A$ the \emph{extent} and $B$ the \emph{intent} of the concept, respectively. 
\end{definition}

By definition, the extent of a concept is precisely the set of all objects which satisfy all the attributes of its intent. Conversely its intent describes precisely the attributes these objects all share. We can define an ordering on concepts by 
\[
(A_1, B_1) \leq (A_2, B_2) :\iff A_1 \leq A_2  \ \ (\iff B_2 \leq B_1)
\]

The key result of FCA is now the following.

\begin{theorem}[\textbf{Basic theorem of formal concept analysis}] \cite{wille1992concept} For any context $\K$, the set of concepts $\clat(\K)$ forms a complete lattice under $\leq$, with 
\begin{align*}
\bigwedge_{i \in I} (A_i, B_i) 
&=
\left(\bigcap_{i \in I} A_i, \cls{\bigcup_{i \in I} B_i}\right)
\\ 
\bigvee_{i \in I} (A_i, B_i) 
&=
\left(\cls{\bigcup_{i \in I} A_i}, \bigcap_{i \in I} B_i\right)
\end{align*}
\end{theorem}

\begin{example}
A formal context is typically depicted in terms of the \emph{cross-table} of the relation $\I$, and the corresponding Hasse diagram of its concept lattice, such as the following.
\[
\scalebox{0.8}{\begin{tikzpicture}
	\begin{pgfonlayer}{nodelayer}
		\node [style=whitedot] (0) at (4.25, -0.5) {};
		\node [style=whitedot] (1) at (6.25, -0.5) {};
		\node [style=whitedot] (2) at (8.75, -0.5) {};
		\node [style=whitedot] (3) at (10.75, -0.5) {};
		\node [style=whitedot] (4) at (10.75, -2) {};
		\node [style=whitedot] (5) at (8.75, -2) {};
		\node [style=whitedot] (6) at (6.25, -2) {};
		\node [style=whitedot] (7) at (4.25, -2) {};
		\node [style=whitedot] (8) at (7.5, 1) {};
		\node [style=whitedot] (9) at (7.5, -3.5) {};
		\node [style=none] (10) at (11.25, 0) {juvenile};
		\node [style=none] (11) at (8.75, 0) {canine};
		\node [style=none] (12) at (6.25, 0) {feline};
		\node [style=none] (13) at (3.75, 0) {mature};
		\node [style=none] (14) at (3.75, -2.5) {cat};
		\node [style=none] (15) at (6.25, -2.5) {dog};
		\node [style=none] (16) at (8.75, -2.5) {kitten};
		\node [style=none] (17) at (11.25, -2.5) {puppy};
		\node [style=none] (18) at (-7.25, -1.25) {\begin{tabular}{c|ccccc}
      & Mature & Feline & Canine & Juvenile &  \\
      \hline
Cat  & X      & X        &       &      &  \\
Dog & X      &          & X     &      &  \\
Kitten   &        & X        &       & X    &  \\
Puppy   &        &          & X     & X    & 
\end{tabular}};
	\end{pgfonlayer}
	\begin{pgfonlayer}{edgelayer}
		\draw (9) to (7);
		\draw (9) to (6);
		\draw (9) to (5);
		\draw (9) to (4);
		\draw (4) to (3);
		\draw (7) to (0);
		\draw (6) to (2);
		\draw (5) to (1);
		\draw (1) to (8);
		\draw (8) to (2);
		\draw (3) to (8);
		\draw (8) to (0);
		\draw (0) to (6);
		\draw (7) to (1);
		\draw (2) to (4);
		\draw (5) to (3);
	\end{pgfonlayer}
\end{tikzpicture}
}
\]
Here $\G = \{\text{Cat}, \text{Dog}, \text{Kitten}, \text{Puppy}\}$ while $\M = \{\text{Mature}, \text{Feline}, \text{Canine}, \text{Juvenile}\}$.
\end{example}

\begin{example} \label{ex:SXA}
Any set $A$ determines a formal context 
\[
\Sx{A} := (A, A, \neq)
\]
Here every subset $B \subseteq A$ is closed, with $B' = A \setminus B$, so that $\clat(\Sx{A}) \simeq \pset(A)$, the power set of $A$. In particular we define the \emph{trivial context} to be $\Ictx := \Sx{\{\star\}}$.
\end{example}

\begin{example} \label{ex:DM}
For any partially ordered set $P$ we can define a formal context 
\[
F(P) := (P,P,\leq)
\]
The lattice $\clat(F(P))$ is the smallest complete lattice in which $P$ can be order-embedded, known as the \emph{Dedekind-MacNeille completion} of $P$ \cite[p 48]{ganter1999formal}. In particular, when $V$ is a complete lattice we have an isomorphism $V \simeq \clat(F(V))$. Thus every complete lattice arises as a concept lattice.
\end{example}

\begin{example}
Any Hilbert space $\hilbH$ determines a formal context
\[
(\hilbH, \hilbH, \bot)
\]
where $\bot$ is its orthogonality relation. The concept lattice of this context is isomorphic to the orthomodular lattice of subspaces $V \leq \hilbH$, via $V \mapsto (V,V^\bot)$. 
\end{example}

\subsection{Notation}

We will shortly describe morphisms of formal contexts based on relations and so fix some conventions about these. For any sets $A, B$ and any relation $R \colon A \to B$, meaning a subset $R \subseteq A \times B$, we denote the converse relation by $R^\dagger \colon B\ \to A$. For each subset $X \subseteq A$ we set 
\[
R(X) := \{b \in B \mid (\exists x \in X) R(x,b)\}
\]
We will often equate $R$ with its induced mapping $A \to \pset(B)$, and so define $R$ by specifying the subsets $R(a) := R(\{a\}) \subseteq B$ for each $a \in A$. The map $X \mapsto R(X)$ has an adjoint $R^\bullet \colon \pset(B) \to \pset(A)$ given by 
\[
R^\bullet(Y) 
:= 
\{a \in A \mid R(a) \subseteq Y\}
\]
for each $Y \subseteq B$. Finally, we also define a map $R_\ua \colon \pset(A) \to \pset(B)$ by 
\begin{equation}
R_\ua(X) := \{b \in B \mid (\forall x \in X) \ R(x,b)\}
\end{equation}
for each $X \subseteq A$.

\section{A Category of Formal Contexts} \label{sec:Cxt}

 We now introduce morphisms of contexts. In fact we will give four equivalent ways of describing such a morphism, with most of the results of this section being essentially due to Mori who studied these maps in \cite{mori2008chu}. Throughout, let $\K_1 = (\G_1, \M_1, \I)$, $\K_2 = (\G_2, \M_2, \I)$, $\dots$ be contexts. 

\begin{definition}
In the category $\Cxt$, the objects are formal contexts $\K$ and the morphisms $\K_1 \to \K_2$ are relations $R \colon G_1 \to G_2$ which are \emph{closed}, meaning that
\begin{enumerate}
\item $R(g) \subseteq G_2$ is closed, for all $g \in G_1$;
\item 
$R^\bullet \colon \pset(\G_2) \to \pset(\G_1)$ preserves closed sets. 
\end{enumerate}
The composition of $R \colon \K_1 \to \K_2$ and $S \colon \K_2 \to \K_3$ is defined by
\[
(S \circ R)(g) := \cl{(S(R(g))} \quad (\forall g \in \G_1)
\]
The identity morphism on $\K$ is the relation $g \mapsto \cl{g}$ for all $g \in G$. 
\end{definition}

To establish that $\Cxt$ is a valid category, we will use the following.

\begin{lemma} \label{lem:cxt-lem}
% Let $\K_1, \K_2, \K_3$ be formal contexts. 
% \begin{enumerate}
% \item \label{enum:closed-reln-alt}
A relation $R \colon G_1 \to G_2$ is closed iff $R(g) \subseteq G_2$ is closed for all $g \in G_1$, and for all subsets $A \subseteq G_1$ we have 
\begin{equation} \label{eq:closure-char}
\cls{R(A)} = \cls{R(\cls{A})}
\end{equation}
In fact for all $R \colon \K_1 \to \K_2$ and $S \colon \K_2 \to \K_3$ in $\Cxt$ and $A\subseteq G_1 $ we have 
\begin{equation} \label{eq:comp-rule} 
\cl{(S \circ R)(A)}
=
\cl{S(R(A))}
\end{equation}
% \end{enumerate}
\end{lemma}
\begin{proof}
% \ref{enum:closed-reln-alt}: 

For the first point, note that for any closed relation $R \colon G_1 \to \G_2$ and $A \subseteq G_1$, $B \subseteq M_2$ we have 
\begin{align*}
R(A) \I B 
&\iff R(A) \subseteq B' \\ 
&\iff A \subseteq R^\bullet(B') \\ 
&\iff \cls{A} \subseteq R^\bullet(B') 
\iff R(\cls{A}) \I B
\end{align*}
using that $R^\bullet(B')$ is closed since $B'$ is. It follows that \eqref{eq:closure-char} holds. Conversely if this is the case then for all $A \subseteq G_1$ and $B \subseteq M_2$ we have 
\begin{align*}
A \subseteq R^\bullet(B')
&\iff R(A) \subseteq B' \\ 
&\iff \cls{R(A)} \subseteq B' \\ 
&\iff \cls{R(\cls{A})} \subseteq B' 
\iff \cls{A} \subseteq R^\bullet(B')
\end{align*}
using that $B'$ is closed in the second step, and so $R^\bullet(B')$ is closed as required.

For \eqref{eq:comp-rule} note that by definition
\begin{align*}
\cls{(S \circ R)(A)} = \cls{\bigcup_{a \in A} \cls{S(R(a))}}
= \cls{\bigcup_{a \in A} S(R(a))} = \cls{S(R(A))}
\end{align*}
using that in any context closure operators satisfy $\cls{\bigcup_{i \in I} A_i} = \cls{\bigcup_{i \in I} \cls{A_i}}$.
\end{proof}

\begin{corollary} $\Cxt$ is a well-defined category.
\end{corollary}
\begin{proof}
That composite of $R \colon \K_1 \to \K_2$ and $S \colon \K_2 \to \K_3$ is indeed a closed relation by Lemma \ref{lem:cxt-lem} since for all $A \subseteq G_1$ we have 
\[
\cls{(S \circ R)(A)} = \cls{S(R(A)} = \cls{S(\cls{R(A)})} = \cls{S(\cls{R(\cls{A})})}
= \cls{S(R(\cls{A}))} = \cls{(S \circ R)(\cls{A})}
\]
We further have $R \circ \id{} = R$ since $R(g) = \cls{R(g)} = \cls{R(\cls{g})}$ for all $g \in G$ and $\id{} \circ R = R$ again follows from Lemma \ref{lem:cxt-lem}. 
\end{proof}

Though our definition of morphism refers only to objects, and not attributes, 
we see shortly that each closed relation $R \colon G_1 \to G_2$ is equivalently described by another $R^* \colon M_2 \to M_1$ in the opposite direction, related to $R$ in the following manner studied by Mori.

\begin{definition}\cite{mori2008chu}
A \emph{Chu correspondence} $(R, S) \colon \K_1 \to \K_2$ is a pair of relations $R \colon \G_1 \to \G_2$ and $S \colon \M_2 \to M_1$ for which each of the sets $R(\g_1) \subseteq \G_2$ and $S(\m_2) \subseteq \M_1$ are closed and we have 
\[
R(\g_1) \I \m_2 \iff \g_1 \I S(\m_2)
\]
for all $\g_1 \in \G_1$, $\m_1 \in M_2$. 
\end{definition}

Another notion of morphism  of contexts was put forward by Ganter and Wille directly in the context of formal concept analysis \cite{ganter1999formal}.

\begin{definition}
A \emph{bond} is a relation $B \colon \G_1 \to \M_2$ for which $B(\g_1) \subseteq \M_2$ and $B^\dagger(\m_2) \subseteq \G_1$ are closed, for all $\g_1 \in \G_1$ and $\m_2 \in \M_2$. 
\end{definition}

Thus a bond is simply a relation from $\G_1$ to $\M_2$ whose rows and columns are closed. We can now show that all of these notions of morphism are equivalent. 

\begin{proposition} \label{prop:morphisms-the-same}
For any contexts $\K_1, \K_2$ there are bijections between:
\begin{enumerate}
\item  \label{enum:ex-rel}
Closed relations $R \colon G_1 \to G_2$;
% Extent relations $_\colon \G_1 \to \G_2$;
\item \label{enum:in-rel}
Closed relations $R^* \colon M_2 \to M_1$; 
% Intent relations $R \colon \M_2 \to \M_1$;
\item \label{enum:chu}
Chu correspondences $(R, R^*) \colon \K_1 \to \K_2$;
\item \label{enum:bond}
Bonds $B \colon G_1 \to M_2$;
\end{enumerate}
given by 
\begin{equation} \label{eq:relate-L-R-B}
R(g_1) = {R^*}^\bullet(g_1')' = B(g_1)' \qquad
R^*(m_2) = R^\bullet(m_2')'  \qquad B(g_1) = R(g_1)'\\ 
\end{equation}
for each $g_1 \in G_1$ and $m_2 \in M_2$.
\end{proposition}

\begin{proof}
\ref{enum:ex-rel} $\iff$ \ref{enum:in-rel} $\iff$ \ref{enum:chu}.
For any closed relation $R \colon \G_1 \to \G_2$ define $R^* \colon \M_2 \to \M_1$ as above. Since each set $R^\bullet(m_2')$ is closed we have $R^\bullet(m_2') = R^\bullet(m_2')'' = R^*(m_2)'$ and so 
\begin{align*}
R(g_1) \I m_2 &\iff R(g_1) \subseteq m_2' \\ 
&\iff g_1 \in R^\bullet(m_2') = R^*(m_2)'
\\ &\iff g_1 \I R^*(m_2)
\end{align*}
and so $(R, R^*)$ is a Chu correspondence. 

Conversely, let $(R, S)$ be any Chu correspondence. Then it is easy to see that  $R^\bullet(B') = S^\bullet(B)'$ for all $B \subseteq M_2$. Then if $A \subseteq \G_2$ is closed we have 
\[
R^\bullet(A) = R^\bullet(A'') = S(A')'
\]
and so $R^\bullet(A)$ is closed. Hence $R$ is a closed relation (and similarly so is $S$). We now verify that $S = R^*$. But by definition $g_1 \I S(m_2)$ iff $R(g_1) \I m_2$ iff $R(g_1) \subseteq m_2'$ iff $g_1 \in R^\bullet(m_2')$. Similarly one may see that $R(g_1) = S^\bullet(g_1')'$ for all $g_1 \in G_1$, i.e. $R=R^**$.

\ref{enum:chu} $\iff$ \ref{enum:bond}
Let $(R, S)$ be a Chu correspondence and define $B \subseteq \G_1 \times M_2$ by $B(g_1) = R(g_1)'$. By construction each $B(\g_1)$ is closed. Now by definition $g_1 \in B^\dagger(\m_2)$ whenever $R(g_1) \I m_2$. But this holds iff $g_1 \I S(m_2)$ iff $g_1 \in S(m_2)'$. Hence $B^\dagger(m_2) = S(m_2)'$, making it closed, so $B$ is a bond.

Conversely, suppose $B$ is a bond and define $R, S$ as above. By construction $R(\g_1)$ and $S(\m_2)$ are closed and we have 
\begin{align*}
R(g_1) \I m_2 &\iff B(g_1)' \I m_2 
\\ &\iff m_2 \in \cl{B(g_1)} = B(g_1)
\\ &\iff g_1 \in B^\dagger(m_2) = \cl{B^\dagger(m_2)} 
\\ &\iff g_1 \I B^\dagger(m_2)' = S(m_2) 
\end{align*}
making $(R, S)$ a Chu correspondence. Since $R(g_1) = R(g_1)''$, $S(\m_2)=S(m_2)''$ and $B(\g_1)=B(g_1)''$ for any Chu correspondence $(R, S)$ or bond $B$, the assignments $(R, S) \leftrightarrow B$ are inverse. 
\end{proof}

Each of the above correspondences may be made functorial. Firstly, for any context $\K = (G, M, \I)$ define the \emph{dual context}
\[
\K^* := (M, G, \I^{\dagger})
\]
by swapping objects and attributes. Let us say that a category $\catC$ is \emph{self-dual} when it comes with an equivalence $(-)^* \colon \catC^\op \simeq \catC$ satisfying $A^{**} = A$ for all objects $A$ and $f^{**} = f$ for all morphisms $f$. 

\begin{lemma}
The assignment $\K \mapsto \K^*$ and $R \mapsto R^*$ defines a self-duality $(-)^* \colon \Cxt^\op \simeq \Cxt$. 
\end{lemma}
\begin{proof}
 For any $R \colon \K_1 \to \K_2$, Proposition \ref{prop:morphisms-the-same} tells us that $R^*$ is the unique morphism for which $(R,R^*)$ forms a Chu correspondence. It follows easily that $\id{\K}^* = \id{\K}$ and that $R^{**} = R^*$ since $(R^*,R)$ is a Chu correspondence. Moreover if $(R_1,S_1)$ and $(R_2,S_2)$ are Chu correspondences one may see verify that $(R_2 \circ R_1, S_1 \circ S_2)$ is also, and so $(-)^*$ preserves composition.
 \end{proof}

Proposition \ref{prop:morphisms-the-same} also shows that $\Cxt$ is isomorphic to the category $\ChuCors$ of Chu correspondences studied in the article \cite{mori2008chu}, where the latter is also shown to be isomorphic to the category $\Bonds$ in which morphisms $\K_1 \to \K_2$ are bonds $B \colon G_1 \to M_2$, under the composition rule
\[
(B_2 \circ B_1)(g) := (B_2)_\ua(B_1(g)') \quad (\forall g \in \G_1)
\]
with the identity bonds being the relations $\I$. We verify this result ourselves.

\begin{lemma}
There is an isomorphism of categories
$\Cxt \simeq \Bonds$.
\end{lemma} 
\begin{proof}
We will use the correspondence of Proposition \ref{prop:morphisms-the-same}.

We first establish the following fact. For any closed relation $R$ the bond $B$ of Proposition \ref{prop:morphisms-the-same} satisfies 
\begin{equation} \label{eq:up-useful}
R(A)' = \bigcap_{a \in A} R(a)' = \bigcap_{a \in A} B(a) = B_\ua(A) 
\end{equation}
for all $A \subseteq \G_1$. 
Now for any morphisms $R_1 \colon \K_1 \to \K_2$ and $R_2 \colon \K_2 \to \K_3$ with corresponding bonds $B_1, B_2$ we have 
\begin{align*}
(B_2 \circ B_1)(g) 
:=
(B_2)_\ua(B_1(g)') 
&=
(B_2)_\ua(R_1(g)'')
=
(B_2)_\ua(R_1(g))=
R_2(R_1(g))'
\\
&=
(R_2(R_1(g)))'''
=
(R_2 \circ R_1)(g)'
\end{align*}
as required.
\end{proof}

We note also that Moshier has described a seemingly alternative relational category of formal contexts \cite{moshier2016relational}, further studied by Jipsen \cite{jipsen2012categories}. In fact this category coincides with our own. 

\begin{lemma} \label{lem:Moshier}
The category $\Cxt$ is identical to that of the same name in \cite{moshier2016relational,jipsen2012categories}. In particular their `compatible relations' are precisely bonds. 
\end{lemma}

\begin{proof} 
Appendix \ref{sec:appendix}. 
 \end{proof}

\section{Monoidal Structures on Formal Contexts} \label{sec:Cxt-Mon}

We will now define two distinct monoidal structures on $\Cxt$, each sharing the same tensor unit $\Icxt$, but with different tensor operations.

The first tensor has been described in the context of Chu correspondences \cite{mori2008chu}, and is motivated by its close connection to the tensor product of sup-lattices, as we see in Section \ref{sec:lattices}.

\begin{definition}
For any contexts $\K_1, \K_2$ we define their \emph{lattice tensor} as
\[
\K_1 \lotimes \K_2 := 
(G_1 \times G_2, \Cxt(\K_1, \K_2^*), \I)
\]
where for any morphism $\K_1 \to \K_2^*$ corresponding to a relation $R \colon \G_1 \to \G_2$ we set $(g_1,g_2) \I R$ whenever $R(g_1) \I g_2$. Equivalently, we have that $(g_1, g_2) \in B$ where $B$ is the bond induced by $R$. 
\end{definition}

Another tensor of contexts has been introduced by Wille directly for FCA. 

\begin{definition}
For any contexts $\K_1, \K_2$ we define their \emph{concept tensor} as 
\[
\K_1 \fcotimes \K_2 
:= 
(\G_1 \times \G_2,
\M_1 \times \M_2,
\triangledown)
\]
where 
\[
(\g_1, \g_2) \triangledown (\m_2, \m_2) \iff \g_1 \I \m_1 \text{ or }\g_2 \I \m_2 
\]
In \cite{wille1985tensorial} this is called the \emph{direct product} of contexts, and denoted $\K_1 \times \K_2$.  
\end{definition}

Since both $\lotimes$ and $\fcotimes$ are defined in the same way on the extent parts of a context, we can in fact describe their bifunctors and structure isomorphisms in the same way. We do so explicitly for $\fcotimes$. For any morphisms $R_1 \colon \K_1 \to \K_3$ and $R_2 \colon \K_2 \to \K_4$ we define $R_1 \otimes R_2 \colon \K_1 \otimes \K_2 \to \K_3 \otimes \K_4$ by
\begin{equation} \label{eq:fc-tensor-Chu}
(R_1 \fcotimes R_2) (g_1, g_2) = \cls{R_1(g_1) \times R_2(g_2)} \subseteq \G_3 \times \G_4 
\end{equation}
for $(g_1,g_2) \in \G_1 \times \G_2$. We define the structure isomorphisms
\[
\begin{tikzcd}
\K_1 \fcotimes (\K_2 \otimes \K_3) \rar{\alpha} & (\K_1 \fcotimes \K_2) \fcotimes \K_3
\end{tikzcd}
\]
\[
\begin{tikzcd}
\K_1 \fcotimes \K_2 \rar{\sigma} & \K_2 \fcotimes \K_1
\end{tikzcd}
\qquad
\begin{tikzcd}
\K_1 \fcotimes \Ictx \rar{\rho} & \K_1
\end{tikzcd}
\]
by 
\begin{align*}
\alpha(g_1,(g_2,g_3)) &= \cls{((g_1,g_2),g_3)} &
\sigma(g_1,g_2) &= \cls{(g_2,g_1)} &
\rho(g_1,\star) &= \cls{g_1} 
\end{align*}
where $g_i \in G_i$ for $i=1,2,3$. 
In other words, the extent parts of coherence isomorphisms are just like those of $\Rel$, but then followed by the closure operator. The bifunctor and coherence isomorphisms for $\lotimes$ are given in the same way, swapping the symbol $\fcotimes$ with $\lotimes$.

\begin{theorem} \cite{mori2008chu}
$(\Cxt, \lotimes, \Ictx)$ is a symmetric monoidal category. 
\end{theorem}

Let us now verify the new result that $\fcotimes$ yields a monoidal structure also. We begin with some straightforward results about the tensor. 

\begin{lemma} \label{lem:tens-help}
For any $A\subseteq \G_1$ and $B\subseteq \G_2$, in $\K_1 \fcotimes \K_2$ we have
\begin{enumerate}
\item $A\times B \I C \times D \iff A \I C$ or $B \I D$;
\item $\cls{A \times B} = \cls{\cls{A} \times \cls{B}} = \cls{A} \times \cls{B} \cup (M_1 \times M_2)'$ 
\item 
$\cls{(R_1 \fcotimes R_2)(A \times B)} = \cls{R_1(A) \times R_2(B)}$ for all closed relations $R_1, R_2$.
\end{enumerate}
\end{lemma}

\begin{theorem} \label{thm:fc-SMC-cxt}
$(\Cxt, \fcotimes, \Ictx)$ is a symmetric monoidal category. 
\end{theorem}

\begin{proof}
Firstly, \eqref{eq:fc-tensor-Chu} forms a Chu correspondence with the relation
\[
(R_1 \fcotimes R_2) (m_3, m_4) = \cls{R_1(m_3) \times R_2(m_4)} 
\]
since
\begin{align*}
 \cls{R_1(g_1) \times R_2 (g_2)} \I (m_3, m_4) 
&\iff R_1(g_1) \times R_2 (g_2) \I (m_3, m_4) 
\\ &\iff R_1(g_1) \I m_3 \text{ or } R_2(g_2) \I m_4
\\ &\iff g_1 \I R_1(m_3) \text{ or } g_2 \I R_2(m_4)
\\ &\iff (g_1, g_2) \I \cls{R_1(m_3) \times R_2(m_4)}
\end{align*}
To see that $\fco$ preserves identities, note that
\[
\id{} \fco \id{} (g_1, g_2) = \cls{\cls{g_1} \times \cls{g_2}} = \cls{(g_1,g_2)} = \id{}(g_1,g_2)
\]
Moreover $\fco$ is a bifunctor since 
\begin{align*}
(R_3 \fco R_4) \circ (R_1 \fco R_2) (g_1, g_2)
&= (R_3 \fco R_4)(\cls{R_1(g_1) \times R_2(g_2)})
\\&= \cls{(R_3 \fco R_4)(\cls{R_1(g_1) \times R_2(g_2)})} 
\\&= \cls{(R_3 \fco R_4)(R_1(g_1) \times R_2(g_2))}
\\&= \cls{R_3(R_1(g_1)) \times R_4(R_2(g_2))}
\\&= (R_3 \circ R_1) \fcotimes (R_4 \circ R_2) (g_1, g_2)
\end{align*}
where in the second step we used that the result will be closed as $(R_3 \fco R_4) \circ (R_1 \fco R_2) $ is a closed relation. 

It is straightforward to verify that  the coherence isomorphisms $\alpha, \rho, \sigma$ are valid morphisms and are isomorphisms with inverses defined element-wise in terms of those of $\Rel$, followed by closure operators. We verify naturality of $\alpha$, while naturality of $\rho$ and $\sigma$ are simpler. Using Lemma \ref{lem:tens-help} one may check that  
\[
R_1 \fco (R_2 \fco R_3)(g_1(g_2,g_3)) = \cls{R_1(g_1) \times (R_2(g_2) \times R_3)}
\]
and also $\cls{\alpha(A\times (B\times C)} = \cls{A \times B) \times C}$ for all subsets $A_i$ of $G_i$. It follows that 
\[
\alpha \circ (R_1 \fco (R_2 \fco R_3))(g_1,(g_2,g_3)) = \cls{(R_1(g_1) \times R_2(g_2)) \times R_3(g_3)}
\]
which is straightforwardly seen to be equal to $(R_1 \fco R_2) \fco R_3) \circ \alpha \circ (g_1, (g_2, g_3))$. Hence $\alpha$ is natural. The coherence equations may be verified by using Lemma \ref{lem:tens-help} to reduce to the usual coherence equations in $\Rel$, followed by applying closure operators once at the end. 
\end{proof}

% \subsection{Further structure on $\Cxt$}

% The category $\Ctx$ comes with further structure still. Firstly, it is \emph{self-dual}, coming with an involutive equivalence
% \[
% (-)^* \colon \Ctx^\op \simeq \Ctx
% \]
% which sends each context $\K = (G, M, \I)$ to the \emph{dual context}
% \[
% \K^* := (M, G, \I^{\dagger})
% \]
% and each morphism $R \colon \K_1 \to \K_2$ to $R^*\colon \K_2^* \to \K_1^*$ as defined in Proposition \ref{prop:morphisms-the-same}, which also showed that $R^{**} =R$.
The category $\Ctx$ comes with further structure still. Recall that a symmetric monoidal category $\catC$ is said to have \myemph{discarding} when each object $A$ comes with a chosen morphism $\discard{A} \colon A \to I$, such that $\discard{I} = \id{I}$ and $\discard{A \otimes B} = \lambda \circ (\discard{A} \otimes \discard{B})$. For example, $\Rel$ has discarding with $\discard{A}$ being the relation with $a \mapsto \star$ for all $a \in A$. 

%Todo: Draw a conclusion about whether we can extend this to \lotimes. 
\begin{proposition} \label{prop:further-structure}
 \ 
\begin{enumerate}
\item \label{enum:tens}
% $(-)^*$ provides a strong monoidal equivalence $(\Ctx, \fco)^\op \simeq (\Cxt, \fco)$.
For all contexts $\K_1, \K_2$ we have 
\[
(\K_1 \fco \K_2)^* := \K_1^* \fco \K_2^*
\]
Hence the equivalence $(-)^*$ is strong monoidal with respect to $\fco$. 
\item \label{enum:SM-disc}
$(\Ctx,\fco)$ is a symmetric monoidal category with discarding. 
\item \label{enum:embed}
There is a full and faithful strong monoidal functor 
% \[
$(\Rel, \times) \hookrightarrow (\Ctx, \fco)$
% \]
which preserves discarding and maps $(-)^\dagger$ to $(-)^*$. 
\end{enumerate}
\end{proposition}
\begin{proof}
\ref{enum:tens} is immediate from the definitions. For \ref{enum:SM-disc}, on each object $\K$ we set $\discard{\K}$ to have intent relation $R$ satisfying $R(\star) = M$. Then $\discard{\Ictx} = \id{}$, and we have $\discard{\K_1 \fco \K_2}(\star) = \cls{M_1 \times M_2} = M_1 \times M_2$. 

For \ref{enum:embed}, via Example \ref{ex:SXA} we define the embedding by $A \mapsto \Sx{A}$ and viewing each relation $R \colon A \to B$ as a closed relation. By construction $\Sx{A} = \Sx{A}^*, \Sx{A \times B} = \Sx{A} \fcotimes \Sx{B}$ and one may verify that $\Sx{R}^* = \Sx{R^\dagger}$ for each relation $R$. 
\end{proof}

However, $\lotimes$ is even more well-behaved, in the following sense. Recall that a self-dual symmetric monoidal category $\catC$ is \emph{$^*$-autonomous} when it comes with natural isomorphisms $\catC(A \otimes B, C^*) \simeq \catC(A, (B \otimes C)^*)$. 

\begin{theorem} \cite{mori2008chu}
$(\Cxt, \lotimes, \Ictx)$ is a $^*$-autonomous category. 
\end{theorem}

On the other hand, since $(\Cxt, \lotimes)$ is not compact closed, it follows that $\K_1^* \lotimes \K_2^* \not \simeq (\K_1 \lotimes \K_2)^*$. 

\section{Categories of Lattices} \label{sec:lattices}

We now study how our category $\Cxt$ and its monoidal structures relate to those of complete lattices, via the concept lattice construction. 

Throughout, we write $\SupLat$ for the category of complete sup-lattices. That is, the objects are complete lattices (with $0, 1$) and the morphisms are mappings $f \colon V_1 \to V_2$ which preserve arbitrary suprema. Similarly we write $\InfLat$ for the category of complete inf-semilattices. 

There is an isomorphism of categories $\SupLat \simeq \InfLat$ given by simply switching $\leq$ with $\geq$. Moreover, both categories are self-dual, with 
\[
(-)^* \colon \SupLat^\op \simeq \SupLat
\]
sending each lattice $V = (V,\leq)$ to the opposite lattice $V^* = (V,\geq)$ and $f \colon V \to W$ to its adjoint $f^* \colon W^* \to V^*$. Since $f$ preserves suprema, $f^*$  preserves infima $W \to V$ and hence suprema $W^* \to V^*$. We denote the 2-element complete lattice by $\TV := \{0 \leq 1\}$.

 Recall that any context $\K$ defines its concept lattice lattice $\clat(\K)$ and any complete lattice $V$ defines a context $F(V)$ via Example \ref{ex:DM}. A key fact is the following, which is essentially from \cite{ganter1999formal}, and more explicitly in \cite{mori2008chu}.

\begin{theorem} \label{thm:equiv-cat-level}
There is a $(-)^*$-preserving equivalence of categories 

\begin{equation} \label{eq:lattice-equiv}
\begin{tikzcd}
\Cxt \arrow[rr,"\clat{(-)}", bend left=15] & \simeq & \SupLat
\arrow[ll,"F", bend left=15]
\end{tikzcd}
\end{equation}
\end{theorem} 
\begin{proof}
For each morphism $R \colon \K_1 \to \K_2$ we define a join preserving map $\clat(R) \colon \clat(\K_1) \to \clat(\K_2)$ by
\begin{equation} \label{eq:suplatfuncdef}
\clat(R)(A,A') := (\cl{R(A)},R(A)')
\end{equation}
for each $(A, A') \in \clat(\K_1)$. Conversely, for any complete sup-lattice morphism $f \colon V_1 \to V_2$ we define $F(f) \colon F(V_1) \to F(V_2)$ to have by $F(f) = f$ which is indeed a closed relation, forming a Chu correspondence with its order adjoint $f^*$. 

Noting that $R(A)' = B_\ua(A)$ via \eqref{eq:up-useful}, the assignment \eqref{eq:suplatfuncdef} is a bijection on homsets by \cite[Theorem 53, Corollary 112]{ganter1999formal}, and is functorial by \cite[Proposition 113]{ganter1999formal}, as is $F$. Every complete lattice $V$ is readily shown to satisfy $V \simeq \clat(F(V))$, making this an equivalence \cite[Theorem 73]{mori2008chu}. It is easy to check that $\clat(-)^* =\clat((-)^*)$, ensuring that $F$ preserves $(-)^*$ also. 
\end{proof}

In particular it follows that the category $\Cxt$ is complete and co-complete. The following result captures the fact that `states in $\Cxt$ are concepts', giving another description of the functor $\clat(-)$, also from \cite{mori2008chu}. 

\begin{lemma} \label{lem:states-are-concepts}
Each homset $\Cxt(\K_1, \K_2)$ forms a complete lattice under inclusion $\subseteq$ of relations and there are natural isomorphisms 
\begin{align}
\Cxt(\Ictx, -) &\simeq \clat(-)  \label{eq:state-isom}
\\  \Cxt(-, \Ictx) &\simeq \clat(-)^* \label{eq:eff-isom}
\end{align}
\end{lemma}
\begin{proof}
The first statement follows from the point-wise ordering on maps in $\SupLat$. 
Now \eqref{eq:state-isom} sends each closed relation $R \colon \{\star\} \to \G$ to the concept $(R(\star),R(\star)')$, and for \eqref{eq:eff-isom} send each closed relation $S \colon \{\star\} \to \M$ to the concept $(S(\star)',S(\star))$. The details may be checked directly, or using the equivalence \eqref{eq:lattice-equiv} and that $\SupLat(\TV,V) \simeq V$ by sending each $f \colon \TV \to V$ to $f(1)$.
\end{proof}

Thanks to the equivalence \eqref{eq:lattice-equiv} each of our monoidal structures $\lotimes, \fcotimes$ on $\Cxt$ corresponds to a monoidal structure on $\SupLat$, and we now describe each.

\subsection{The Lattice Tensor}

As our naming suggests, the lattice tensor on $\Cxt$ corresponds to the most well-known monoidal structure on $\SupLat$, which in fact makes it a *-autonomous category. 
For any complete lattices $V_1, V_2$ we define $(V_1 \multimap V_2) := \SupLat(V_1, V_2)$, which forms a complete lattice under the point-wise ordering of maps, and then 
\[
V_1 \lotimes V_2 := (V_1 \multimap V_2^*)^*
\]
Alternatively, $V_1 \lotimes V_2$ may be represented as the collection of \emph{bi-ideals} in $V_1, V_2$. The following is well-known. 

\begin{theorem}
$(\SupLat, \lotimes, \TV, *)$ is a *-autonomous category. 
\end{theorem}

Moreover, Mori has established the following. 

\begin{theorem} \cite{mori2008chu}
The functors $(\clat,F)$ yield a *-autonomous equivalence
\[
(\Cxt, \lotimes, \Ictx, *) \simeq (\SupLat, \lotimes, \TV, *)
\]
\end{theorem}

\subsection{The Concept Tensor} \label{subsec:concept-suplat}

Less well-known is the monoidal structure on $\SupLat$ corresponding to the tensor $\fcotimes$ on $\Cxt$, suggested by Wille. 

\begin{definition} \cite{wille1985tensorial}
For a pair of complete lattices $V_1, V_2$ we define their \emph{concept tensor} to be the complete lattice 
\[
V_1 \fcotimes V_2 := \clat(F(V_1) \times F(V_2))
\]
Explicitly, it is the concept lattice of the context with $G = M = V_1 \times V_2$ and relation $\triangledown$ with $(x,y) \triangledown (w,z)$ whenever $x \leq w$ or $y \leq z$. 
\end{definition}

The tensor $V_1 \fcotimes V_2$ has a representation in terms of closed bi-ideals of $V_1, V_2$. However, it can also be worked with directly by making use of a pair of complete lattice embeddings
\begin{equation} \label{eq:tensor-embeddings}
\begin{tikzcd}
V_1 \arrow[r,hook,"\varepsilon_1"] & V_1 \fcotimes V_2 & V_2 \arrow[l,hook',swap,"\varepsilon_2"]
\end{tikzcd}
\end{equation}
From these we define a pair of \myemph{tensorial operations} 
$\ovee, \owedge \colon V_1 \times V_2 \to V_1 \fcotimes V_2$ 
by
\begin{align*}
x_1 \twedge x_2 &= \varepsilon_1(x_1) \wedge \varepsilon_2(x_2) \\ 
x_1 \tvee x_2 &= \varepsilon_1(x_1) \vee \varepsilon_2(x_2) 
\end{align*}
By construction we have 
\begin{align} \label{eq:eps-from-join}
\varepsilon_1(x_1) &= x_1 \twedge 1 = x_1 \tvee 0 &
\varepsilon_2(x_2) &= 1 \twedge x_2  = 0 \tvee x_2
\end{align}

The operations $\twedge$ and $\tvee$ satisfy a number of axioms that allow one to perform calculations in $V_1 \fcotimes V_2$, see \cite[p.83]{wille1985tensorial}. Most notably, the subsets $\ve(V_1)$ and $\ve(V_2)$ generate $V_1 \fcotimes V_2$ as a complete lattice, and are `mutually distributive', in the following sense.

\begin{definition} \cite{ganter1999formal}
We call a pair of subsets $X$ and $Y$ of a complete lattice \emph{mutually distributive} when for all indexed sets of elements $(x_i)_{i \in I} \subseteq X$ and $(y_i)_{i \in I} \subseteq Y$ we have 
\begin{align*}
\bigvee_{i \in I} (x_i \wedge y_i) 
&= \bigwedge_{J \subseteq I} 
\left(\bigvee_{j \in J} x_j \vee \bigvee_{k \in I \setminus J} y_k \right) \\
\bigwedge_{i \in I} (x_i \vee y_i) 
&= \bigvee_{J \subseteq I} 
\left(\bigwedge_{j \in J} x_j \wedge \bigwedge_{k \in I \setminus J} y_k \right) 
\end{align*}
\end{definition}

Wille has studied $\fcotimes$ as a tensor for lattices which does not favour suprema over infima (or vice versa), characterising it with respect to complete homomorphisms \cite[Theorem 2]{wille1985tensorial}. However, we will now see that this characterisation may be extended to completely join-preserving maps, yielding a monoidal structure on $\SupLat$. The following results are new.

\begin{proposition} \label{prop:tensor-main}
Let $f \colon V_1 \to M$ and $g \colon V_2 \to M$ be complete sup-lattice morphisms and suppose that $f(V_1)$ and $g(V_2)$ are mutually distributive in $M$. Then there exists a unique complete sup-lattice morphism $h \colon V_1 \fcotimes V_2 \to M$ with 
\begin{equation} \label{eq:tensor-h}
h(x \twedge y) = f(x) \wedge g(y)
\end{equation}
for all $x \in V_1, y \in V_2$. Moreover, when $f$ and $g$ are complete lattice morphisms, so is $h$. 
\end{proposition}

\begin{proof}
The result and proof is similar to \cite[Theorem 37]{ganter1999formal} which, though stated for complete morphisms, in many places only uses preservation of joins. See the Appendix for details.
\end{proof}

As as a consequence we obtain Wille's characterisation of this tensor.

\begin{corollary} \label{cor:unique-compl} \cite[Thm 2]{wille1985tensorial}
For any complete lattice morphisms $f \colon {V_1 \to M}$ and $g \colon V_2 \to M$  whose images $f(V_1)$ and $g(V_2)$ are mutually distributive, there is a unique complete lattice morphism $h \colon V_1 \fcotimes V_2 \to M$ with
\begin{equation} \label{eq:compose-f-g}
h \circ \varepsilon_1 = f \qquad h \circ \varepsilon_2 = g
\end{equation}
\end{corollary}

We can also now make this tensor into a bifunctor, thanks to the following.

\begin{lemma} \label{lem:main-tensor-result}
For any complete sup-lattice morphisms $f \colon V_1 \to W_1$ and $g \colon V_2 \to W_2$ there is a unique such morphism
\[
(f \fcotimes g) \colon V_1 \fcotimes V_2 \to W_1 \fcotimes W_2
\] 
satisfying 
\[
(f \fcotimes g)(x \twedge y) = f(x) \twedge g(y)
\]
for all $(x,y) \in V_1 \times V_2$. If $f$ and $g$ are complete lattice morphisms, so is $f \fcotimes g$. 
\end{lemma}
\begin{proof}
Apply Proposition \ref{prop:tensor-main} to $\varepsilon_1 \circ f $ and $\varepsilon_2 \circ g$, with $M = W_1 \fcotimes W_2$.
\end{proof}

We are now ready to establish the following. We equip $\SupLat$ with discarding morphisms with $\discard{V} \colon V \to \TV$ sending $x$ to $1$ iff $x \neq 0$. Let us also write $\CompLat$ for the wide subcategory of $\SupLat$ given by the completely join and meet preserving maps.

%Todo: Should really check thesis for what a functor 'preserving discarding' means. Probably not coherence isomorphisms to be causal. But then all isomorphisms here are causal.

\begin{theorem} \label{thm:fctensor-SMC}
$(\SupLat, \fcotimes, \TV)$ is a symmetric monoidal category with discarding, with $\CompLat$ as a symmetric monoidal subcategory. Moreover $(\clat, F)$ provide a symmetric monoidal equivalence 
\begin{equation} \label{eq:new-mon-equiv}
(\Cxt, \fco, \Ictx) \simeq (\SupLat, \fco, \TV)
\end{equation}
which preserves discarding. 
\end{theorem}

A consequence is the following, which we could have verified directly.

\begin{corollary}
The self-duality $(-)^* \colon (\SupLat, \fcotimes)^\op \simeq (\SupLat, \fcotimes)$ is strong monoidal. In particular $(V \fcotimes W)^* \simeq V^* \fco W^*$ for all complete lattices $V, W$.
\end{corollary}

\section{Outlook: Applications} \label{sec:applications}

We close by briefly discussing potential applications of the category $\Cxt$ as a compositional model of natural language meaning, which we hope to expand on in future work. 

Before considering each of our tensors, we note that $\Cxt$ itself naturally models order relations on words. In detail, following \cite{coecke2010mathematical}, we choose a formal context $\K$ to represent each basic word type, e.g. nouns. The semantics of a noun $w$ is then a state $\sem{w} \colon \Ictx \to \K$, which by Lemma \ref{lem:states-are-concepts} corresponds to a concept of $\K$. The ordering on concepts allows one to capture entailment, as is treated using density matrices in \cite{balkir2015distributional}.

\subsection{DisCo in Closed Categories}
The DisCo framework is typically applied to autonomous (i.e. rigid) monoidal categories $\catC$, in which each object comes with a dual object with a `cup' and `cap' \cite{coecke2010mathematical}. Any pregroup forms such a category, allowing one to interpret pregroup grammars in $\catC$. 

Though neither of our monoidal structures on $\Cxt$ is autonomous, a known result due to Lambek tells us that in fact a simpler structure than pregroups is required in practice.

\begin{definition}\cite{lambek1997type} 
A \emph{protogroup} is a partially ordered monoid $(P, \leq, \cdot)$ such that for every $a \in P$ there are chosen elements $a^l, a^r \in P$ with 
\begin{equation} \label{eq:protogroup}
a^l \cdot a \leq 1 \qquad a \cdot a^r \leq 1 \qquad \text{(contractions)}
\end{equation}
$P$ is a \emph{pregroup} when additionally for all $a \in A$ we  have 
\begin{equation} \label{eq:pregroup}
1 \leq a \cdot a^l \qquad 1 \leq a^r \cdot a  \qquad \text{(expansions)}
\end{equation}
\end{definition}

\begin{lemma}[Switching Lemma, \cite{lambek2008word}]
For any terms $t, t'$ in the free pregroup $\mathcal{P}_B$ generated by basic types $B$, if $t \leq t'$ then there exists $t'' \in \mathcal{P}_B$ such that $t \leq t''$ without expansions and $t'' \leq t'$ without contractions. 
\end{lemma}

In particular each inequality $t \leq s$ for the sentence type $s$, determining that a phrase of type $t$ is a valid sentence, requires contractions only, e.g.
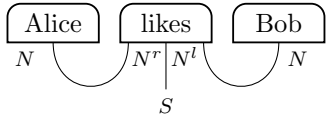
\[
\scalebox{1.0}{\begin{tikzpicture}
	\begin{pgfonlayer}{nodelayer}
		\node [style=none] (1) at (-1, 0.5) {};
		\node [style=none] (2) at (1, 0.5) {};
		\node [style=none] (3) at (-1, 0.5) {};
		\node [style=none] (4) at (1, 0.5) {};
		\node [style=label] (5) at (-1.75, 0) {$N$};
		\node [style=label] (6) at (2, -1.25) {$S$};
		\node [style=wide copoint] (8) at (-1, 1) {Alice};
		\node [style=none] (10) at (2, -0.75) {};
		\node [style=none] (12) at (2, 0.5) {};
		\node [style=label] (15) at (1.5, 0) {$N^r$};
		\node [style=wide copoint] (16) at (2, 1) {likes};
		\node [style=none] (17) at (3, 0.5) {};
		\node [style=none] (18) at (5, 0.5) {};
		\node [style=none] (19) at (3, 0.5) {};
		\node [style=none] (20) at (5, 0.5) {};
		\node [style=label] (21) at (2.5, 0) {$N^l$};
		\node [style=label] (22) at (5.5, 0) {$N$};
		\node [style=wide copoint] (23) at (5, 1) {Bob};
	\end{pgfonlayer}
	\begin{pgfonlayer}{edgelayer}
		\draw (1.center) to (3.center);
		\draw (2.center) to (4.center);
		\draw [bend right=90, looseness=2.00] (1.center) to (2.center);
		\draw (10.center) to (12.center);
		\draw (17.center) to (19.center);
		\draw (18.center) to (20.center);
		\draw [bend right=90, looseness=2.00] (17.center) to (18.center);
	\end{pgfonlayer}
\end{tikzpicture}
}
\]
Hence for such applications protogroup grammars are sufficient, requiring our category to only have `cups' (and not `caps'). 

Any closed symmetric monoidal category $(\catC, \otimes, \multimap)$ can model protogroup grammars as follows. For each object $A$ we set $A^l = A^r = A^* := (A \multimap I)$, and we interpret $a^l \cdot a \leq 1$ as the canonical morphism
$A^* \otimes A \to I$ given by 
\begin{equation} \label{eq:cap-is-eval}
\scalebox{1.0}{\begin{tikzpicture}
	\begin{pgfonlayer}{nodelayer}
		\node [style=none] (1) at (-1, 0.25) {};
		\node [style=none] (2) at (1, 0.25) {};
		\node [style=none] (3) at (-1, 0.5) {};
		\node [style=none] (4) at (1, 0.5) {};
		\node [style=label] (5) at (-1, 1) {$A^*$};
		\node [style=label] (6) at (1, 1) {$A$};
		\node [style=none] (7) at (3, 0) {$:=$};
		\node [style=wide point] (8) at (5.75, -0.5) {$\mathsf{eval}$};
		\node [style=none] (9) at (4.75, 0) {};
		\node [style=none] (10) at (6.75, 0) {};
		\node [style=none] (11) at (4.75, 0.75) {};
		\node [style=none] (12) at (6.75, 0.75) {};
		\node [style=label] (13) at (4.75, 1.25) {$A \multimap I$};
		\node [style=label] (14) at (6.75, 1.25) {$A$};
	\end{pgfonlayer}
	\begin{pgfonlayer}{edgelayer}
		\draw (1.center) to (3.center);
		\draw (2.center) to (4.center);
		\draw [bend right=90, looseness=2.00] (1.center) to (2.center);
		\draw (9.center) to (11.center);
		\draw (10.center) to (12.center);
	\end{pgfonlayer}
\end{tikzpicture}
}
\end{equation}
and $a \cdot a^r \leq 1$ by applying the symmetry to the above.

\subsection{The Lattice Tensor}
%Todo: Need to work out precisely how the cap is indeed defined! 
Since the category $(\Cxt, \lotimes)$ is *-autonomous, it is in particular closed. Hence as above it can model protogroup grammars and all the sentence-parsing aspects of the DisCo framework. We may use the cups defined in \eqref{eq:cap-is-eval} to interpret any valid sentence as a single concept.

\subsection{The Concept Tensor}
In contrast $(\Cxt, \fcotimes)$ is merely a symmetric monoidal category, with no apparent canonical cups or caps. However, Antonin Delpeuch has shown how the DisCo framework may be extended even to bare monoidal categories, thanks to the following result. 

\begin{theorem} \cite{delpeuch2014autonomization}
For any monoidal category $\catC$ there is a free autonomous category $L(\catC)$ with a strong monoidal full and faithful embedding $\catC \hookrightarrow L(\catC)$. 
\end{theorem}

Since the embedding is full, this means that any interpretation of a sentence as a state $I \to S$ involving formal cups and caps can be re-arranged to a valid morphism in $\catC$, and so be rewritten without them. 

In future it would be desirable to fully explore the usefulness of both tensors on $\Cxt$ when modelling sentence meanings. 

\bibliographystyle{alpha}
\bibliography{FCA-bib}

\appendix 

\section{Proofs} \label{sec:appendix}

\begin{proof}[Proof of Lemma \ref{lem:Moshier}]
The morphisms $\K_1 \to \K_2$ in \cite{jipsen2012categories} are relations $B \subseteq G_1 \times M_2$ for which $C = B^\dagger$ satisfies 
\begin{equation} \label{eq:comp-reln}
\cls{C_\ua(Y)} = C_\ua(Y) = C_\ua(\cls{Y})
\end{equation}
for all $Y \subseteq M_2$. We will show that $B$ is a bond. By the above each set $B^\dagger(m_2) = C_\ua(\{m_2\})$ is closed. For any $X \subseteq G_1$ by definition one may see that 
\begin{equation} \label{eq:CB-rule}
X \subseteq C_\ua(Y) \iff Y \subseteq B_\ua(X)
\end{equation}
with either holding iff $B(x,y)$ holds for all $x \in X, y \in Y$. Hence by \eqref{eq:comp-reln} we have $\cls{Y} \subseteq B_\ua(X) \iff Y \subseteq B_\ua(X)$, and so each set $B_\ua(X)$ is closed, making each set $B(g) = B_\ua(\{g\})$ closed as required. 

 Conversely, given a bond $B$ define $C=B^\dagger$ and $R \colon M_2 \to M_1$ its intent relation. Then for all $Y \subseteq M_2$ just as in \eqref{eq:up-useful} one may see that $C_\ua(Y) = R(Y)'$ making each such set closed. Moreover this satisfies \eqref{eq:comp-reln} by Lemma \ref{lem:cxt-lem}. The composition $\bullet$ of $B_1 \colon G_1 \to M_2$ and $B_2 \colon G_2 \to M_3$ in \cite{jipsen2012categories} is
 \[
 (B_2 \bullet B_1)^\dagger(m) := (C_1)_\ua((C_2)_\ua(\{m\})')
 = (R_1)_\ua(R_2(m))'
 = R(m)' = (B_2 \circ B_1)^\dagger(m)
 \]
 for each $m \in M_3$, where $R$ is the intent relation for $(B_2 \circ B_1)$. Hence both categories coincide. 
\end{proof}

\begin{proof}[Proof of Proposition \ref{prop:tensor-main}]
The first part of the proof of  \cite[Theorem 37]{ganter1999formal} shows that for any $A \subseteq V_1 \times V_2$ we have 
\begin{equation}
\bigvee_{(x,y) \in A} f(x) \wedge g(y)
=
\bigwedge_{(w,z) \in A'} f(w) \vee g(z)
\end{equation}
We use this to define $h \colon V_1 \fcotimes V_2 \to M$ by
\begin{equation} \label{eq:meets}
h(A,B) = 
\bigvee_{(x,y) \in A} f(x) \wedge g(y)
=
\bigwedge_{(w,z) \in B} f(w) \vee g(z)
\end{equation}
for each concept $(A, B) \in V_1 \fcotimes V_2$. The verification that $h$ preserves suprema is just as in \cite[Theorem 37]{ganter1999formal}. 

%TODO: Could clarify this a bit more one day. give explicit definition of \twedge has been given somewhere
We now check that \eqref{eq:tensor-h} is indeed satisfied. From the explicit definition of $\twedge$ in \cite{wille1985tensorial} we have that $x \twedge y = (A,B)$ where $(w,z) \in A$ whenever $w \leq x$ and $z \leq y$ or $w = 0$ or $z = 0$. Hence we have
\begin{align*}
h(x \twedge y)  &= \bigvee_{w \leq x, z \leq y} f(w) \wedge g(z) \vee \bigvee_{w \in V_1} f(w) \wedge g(0) \vee \bigvee_{z \in V_2} f(0) \wedge g(z) \\ 
&= (f(x) \wedge g(y)) \vee 0 \vee 0 = f(x) \wedge g(y)
\end{align*}

It remains for us to verify that $h$ is unique. Firstly, let us consider when $f$ and $g$ are complete lattice morphisms, and so preserve infima. In this case, by the symmetry of the definition \eqref{eq:meets}, $h$ does also, making it a complete lattice morphism as stated. Moreover since the subsets $\varepsilon_1(V_1)$ and $\varepsilon_2(V_2)$ generate $V_1 \fcotimes V_2$, $h$ is fully determined by \eqref{eq:tensor-h} and \eqref{eq:eps-from-join}, making it unique. 

In particular, taking $f = \varepsilon_1$ and $g = \varepsilon_2$ we must have $h = \id{V_1 \fcotimes V_2}$ since $x \twedge y = \varepsilon_1(x) \wedge \varepsilon_2(y)$. Now \eqref{eq:meets} tells us that for any $(A, B)$ in $V_1 \fcotimes V_2$ we have 
\begin{equation} \label{eq:concept-as-wedge}
(A,B) = \bigvee_{(x,y) \in A} x \twedge y
\end{equation}
Hence any sup-preserving map $V_1 \fcotimes V_2 \to M$ is determined entirely by its action on elements of the form $x \twedge y$, making $h$ unique in the general case.
\end{proof}

\begin{proof}[Proof of Theorem \ref{thm:fctensor-SMC}]
From the uniqueness in Lemma \ref{lem:main-tensor-result}, $\fcotimes$ preserves identities and composition, making it a bifunctor on $\SupLat$. Using Corollary \ref{cor:unique-compl}, we define $\alpha = \alpha_{U,V,W} \colon (U \fcotimes V) \fcotimes W \to U \fcotimes (V \fcotimes W)$ as the unique complete homomorphism with
\begin{align*}
\alpha \circ \ve_{U \fcotimes V} \circ \ve_U &= \ve_U \\ 
\alpha \circ \ve_{U \fcotimes V} \circ \ve_V &= \ve_{V \fcotimes W} \circ \ve_M \\ 
\alpha \circ \ve_W &= \ve_{V \fcotimes W} \circ \ve_W \\ 
\end{align*}
Similarly, we define $\alpha' \colon U \fcotimes (V \fcotimes W) \to (U \fcotimes V) \fcotimes W$ in the analogous way, and then since $\alpha' \circ \alpha$ preserves each of $\ve_{U \fcotimes V} \circ \ve_U$, $\ve_{U \fcotimes V} \circ \ve_V$ and $\ve_W$ it is the identity by uniqueness. Similarly $\alpha \circ \alpha' = \id{}$, making $\alpha$ an isomorphism. By construction  we have 
\[
\alpha((x \twedge y) \twedge z) = x \twedge (y \twedge z)
\]
for all $x \in U, y\in V, z \in W$. A quick calculation shows that for any $f_i \colon V_i \to W_i$ for $i=1, 2, 3$ we have that $(\alpha_{W} \circ (f_1 \fcotimes f_2) \fcotimes f_3)$ and $(f_1 \fcotimes (f_2 \fcotimes f_3) \circ \alpha_{V})$ are equal on elements of the form $((x \twedge y) \twedge z)$. But such elements are sup-dense in $(V_1 \fcotimes V_2) \fcotimes V_3$,  since elements of the form $x \twedge y$ are sup-dense in $V_1 \fcotimes V_2$ by \eqref{eq:concept-as-wedge}. Hence the $\alpha$ are natural. We define the right unitor 
\[
\rho \colon V \fcotimes \TV \to V
\]
to be the unique complete lattice homomorphism with $\rho \circ \ve_V = \id{V}$ and $\rho \circ \ve_\TV(0) = 0$ and $\rho \circ \ve_\TV(1) = 1$. Then we have $\rho(x \twedge 1) = x$ while $\rho(x \twedge 0) = \rho(0) = 0$ for all $x \in V$. Then one may check that $\rho$ is  an isomorphism, and using elements again that $f \circ \rho_V = \rho_W \circ (f \fcotimes \id{\TV})$ for any $f \colon V \to W$ in $\SupLat$, establishing naturality. The left unitor is defined similarly. The symmetry
\[
\sigma \colon V \fcotimes W \to W \fcotimes V
\]
is the unique complete homomorphism with $\sigma \circ \ve_1 = \ve_2$ and $\sigma \circ \ve_2 = \ve_1$. Equivalently this means that $\sigma(x \twedge y) = y \twedge x$ for all $x \in V$, $y \in W$, and then naturality from the definition of $f \fcotimes g$. Moreover we have $\sigma_{W,V} \circ \sigma_{V,W} = \id{V \fcotimes W}$ due to preservation of $\ve_1$ and $\ve_2$ and so $\sigma$ is a symmetry. 

Verifying the coherence conditions is straightforward using elements and that maps from a tensor are determined by elements of the form $x_1 \twedge x_2 \twedge \dots \twedge x_n$ (after bracketing). For example, the triangle law follows from the fact that 
\begin{align*}
(\rho \fcotimes \id{}) \circ \alpha (x \twedge (y \twedge z)) 
&=
(\rho \fcotimes \id{}) ((x \twedge y) \twedge z)
=
\rho(x \twedge y) \twedge z 
\\
&=
(x \wedge y) \twedge z 
=
x \twedge (y \wedge z)
=
(\id{} \fcotimes \lambda) (x \twedge (y \twedge z)) 
\end{align*}
whenever $y \in \{0,1\}$. Hence $\SupLat$ is a symmetric monoidal category. By construction the coherence maps belong to $\CompLat$, and the bifunctor restricts there by Lemma \ref{lem:main-tensor-result}, making $\CompLat$ a symmetric monoidal subcategory.

We now wish to establish the monoidal equivalence \eqref{eq:new-mon-equiv}. 
For any pair of contexts $\K_1, \K_2$, by definition, the elements of $\clat(\K_1) \fco \clat(\K_2)$ are subsets $C \subseteq \clat(\K_1) \times \clat(\K_2)$ which are closed in the appropriate sense. By (the proof of) \cite[Theorem 26]{ganter1999formal} there is a canonical isomorphism $\phi = \phi_{\K_1, \K_2} \colon \clat(\K_1) \fco \clat(\K_2) \to \clat(\K_1 \times \K_2)$ defined by
\begin{equation} \label{eq:phi-isom}
\phi(C) := \bigvee_{(A_1,A_1'),(A_2,A_2') \in C} (\cls{A_1 \times A_2}, (A_1 \times A_2)')
\end{equation}
for each such $C \subseteq \clat(\K_1) \times \clat(\K_2)$, where $(-)'$ and $\cls{(-)}$ are taken in $\K_1 \times \K_2$.

Now, for any concepts $(A,A') \in \clat(\K_1)$ and $(B,B') \in \clat(\K_2)$, from the explicit definition of $\owedge$, we have that $((A_1, A_1'),(A_2, A_2')) \in (A,A') \twedge (B,B')$ whenever $A_1 \leq A$ and $A_2 \leq B$, or $(A_1, A_2) \in (M_1 \times M_2)'$. It follows that we have 
\[
\phi((A,A') \twedge (B,B')) = (\cls{A \times B}, (A \times B)')
\]
Hence for all morphisms $f \colon V_1 \to W_1$ and $g \colon V_2 \to W_2$ in $\SupLat$ we have 
\begin{align*}
(\clat(f \fco g) \circ \phi_{V_1, V_2})((A,A') \twedge (B,B')) 
&= 
\clat(f \fco g)(\cls{A \times B}, (A \times B)')
\\ &=
\cls{(f \fco g)(\cls{A \times B})}
\\ &=
\cls{\cls{f(A)} \times \cls{g(B)}}
\\ &=
\phi(\cls{f(A)} \twedge \cls{g(B)})
\\ &=
\phi_{W_1,W_2} \circ (\clat(f) \fco \clat(g))((A,A') \twedge (B,B'))
\end{align*}
where we used Lemma \ref{lem:tens-help} several times in the third step. Since elements of the form $(A,A') \twedge (B,B')$ are sup-dense by \eqref{eq:concept-as-wedge} it follows that the isomorphisms $\phi$ are natural. We have  $\clat(\Ictx) \simeq \pset(\{\star\}) \simeq \TV$. We omit the verification of the monoidal coherence equations, from which it follows that $\clat(-)$ is a monoidal functor, and hence the equivalence a monoidal one.

 Finally, on any context $\K$, from \eqref{eq:suplatfuncdef} and the definition of $\discard{\K} = R$ we have that $\clat(\discard{\K})$ maps a concept $(A, A')$ to $0$ iff $R(A)' = R^{*\bullet}(A') = 0$, which holds iff $A' = M$, that is iff $(A,A') = 0$ in $\clat(\K)$. Hence $\clat$ preserves discarding.
\end{proof}

\end{document}